\documentclass[reqno,12pt]{amsart}
\usepackage[margin=1in]{geometry}
\usepackage{amsfonts}
\usepackage[all]{xy}
\usepackage{enumitem}
\usepackage{amssymb,amsmath,url}
\usepackage{setspace}
\usepackage{amsthm}
\usepackage{cite}
\usepackage{color}
\numberwithin{equation}{section}
\usepackage[normalem]{ulem}
\usepackage{graphicx}
\usepackage{lmodern}
\graphicspath{ {images/} }
\usepackage{pdfpages}
\usepackage{mmacells}
\mmaDefineMathReplacement[≤]{<=}{\leq}
\mmaDefineMathReplacement[≥]{>=}{\geq}
\mmaDefineMathReplacement[≠]{!=}{\neq}
\mmaDefineMathReplacement[→]{->}{\to}[2]
\mmaDefineMathReplacement[⧴]{:>}{:\hspace{-.2em}\to}[2]
\mmaDefineMathReplacement{∉}{\notin}
\mmaDefineMathReplacement{∞}{\infty}
\mmaDefineMathReplacement{𝕕}{\mathbbm{d}}
\mmaSet{
	morefv={gobble=2},
	linklocaluri=mma/symbol/definition:#1,
	morecellgraphics={yoffset=1.9ex}
}

\author{Zhenghui Huo}
\title{$L^p$ estimates for the Bergman projection on some Reinhardt domains}
\begin{document}
	\address{Dept. of Mathematics, Washington University in St. Louis, 1 Brookings Dr., St. Louis MO 63130}
	\email{huo@math.wustl.edu}
		\newtheorem{thm}{Theorem}[section]
	\newtheorem{cl}[thm]{Claim}
	\newtheorem{lem}[thm]{Lemma}
		\newtheorem*{Rmk*}{Remark}
	\newtheorem{ex}[thm]{Example}
	\newtheorem{de}[thm]{Definition}
	\newtheorem{co}[thm]{Corollary}
	\newtheorem*{thm*}{Theorem}
	\maketitle
\begin{abstract}
	We obtain $L^p$ regularity for the Bergman projection on some Reinhardt domains. We start with a bounded initial domain $\Omega$ with some symmetry properties and generate  successor domains in higher {dimensions}. We prove: If the Bergman kernel on $\Omega$ satisfies appropriate estimates, then the Bergman projection on the successor is $L^p$ bounded. For example, the Bergman projection on successors of strictly pseudoconvex initial domains is bounded on $L^p$ for $1<p<\infty$. The successor domains need not have smooth boundary nor be strictly pseudoconvex.

\medskip

\noindent
{\bf AMS Classification Numbers}: 32A25, 32A36, 32A07

\medskip

\end{abstract}

\section{Introduction}
The purpose of this paper is to establish $L^p$ regularity for the Bergman projection on certain domains. In \cite{Zhenghui}, {the author} began with an initial domain with certain symmetry properties. From this initial domain {the author} constructed various successor domains and computed (explicitly) the Bergman kernel on them in terms of the Bergman kernel on the initial domain.

Let $\Omega$ be an initial domain {in $\mathbb C^n$}. We consider two kinds of estimates on the Bergman kernel $K_{\Omega}$. {A first estimate} implies $L^p$ regularity of the Bergman projection on $\Omega$. If, also, {a second estimate} holds, then we obtain $L^p$ regularity of the Bergman projection on the successor domain. See Theorem 1.2.
{We use a variant of Schur's Lemma to establish $L^p$ regularity.  We state the crucial estimates in Theorem 3.3 and give the proof in Section 4.}

Let $\Omega\subseteq \mathbb C^n$ be a bounded domain. The Bergman projection is the orthogonal projection from $L^2(\Omega)$ onto the closed subspace of square-integrable holomorphic functions, and thus is bounded on $L^2$.
It is natural to ask when {this operator} is bounded on $L^p$ for $p\neq 2$. Using known estimates for the Bergman kernel, various authors have obtained $L^p$ regularity results for $1<p<\infty$ in the following settings:
\begin{enumerate}
	\item $\Omega$ is bounded, smooth, and strongly pseudoconvex. See \cite{Fefferman,PS}.
	\item $\Omega\subseteq\mathbb C^2$ is a domain of finite type. See \cite{McNeal1,McNeal3,NRSW}.
	\item $\Omega\subseteq\mathbb C^n$ is a convex domain of finite type. See \cite{McNeal3,McNeal2,MS}.
	\item $\Omega\subseteq\mathbb C^n$ is a domain of finite type with locally diagonalizable Levi form. See \cite{CD}.
\end{enumerate}
Progress has also been made on some domains with weaker assumption on boundary regularity. {In some cases,  the Bergman projection is $L^p$ bounded for $1<p<\infty$, See \cite{EL,LS}. For other domains, the projection has only a finite range of mapping regularity. See \cite{Yunus,DebrajY,EM,EM2,CHEN}. There are also smooth bounded domains where the projection has limited $L^p$ range. See \cite{BS}.} 

We start with a bounded {complete Reinhardt} domain $\Omega$ in $\mathbb C^n$ with a defining function $\rho$, and analyze the $L^p$ regularity of the Bergman projection on the successor domains {$U^{\alpha}(\Omega)$} defined by 
\begin{equation}\label{00}
U^{\alpha}(\Omega)=\left\{(z,w)\in \mathbb C^{n}\times \mathbb B^k:\left(\frac{z_1}{(1-\|w\|^2)^{\frac{\alpha_1}{2}}},\dots,\frac{z_n}{(1-\|w\|^2)^{\frac{\alpha_n}{2}}}\right)\in \Omega\right\}.
\end{equation}
{Here $\mathbb B^k$ is the unit ball in $\mathbb C^k$ and $\alpha=(\alpha_1,\cdots,\alpha_n)$ with each $\alpha_j$ greater than 0. We will often use $U^{\alpha}$ to denote $U^{\alpha}(\Omega)$.}

 For each multi-index $\beta$,  let $D^{\beta}_z$ denote the differential operator $(\frac{\partial}{\partial z_1})^{\beta_1}\cdots (\frac{\partial}{\partial z_n})^{\beta_n}$. {Given} {functions} of several variables $f$ and $g$, we use $f\lesssim g$ to denote that $f\leq Cg$ for a constant $C$. If $f\lesssim g$ and $g\lesssim f$, then we say $f$ is comparable to $g$ and write $f\simeq g$. 
 
 {Next we introduce the estimates needed for the derivatives of the Bergman kernel on $\Omega$.}
\begin{de}
Let $\Omega$ be a domain in $\mathbb C^n$. Let $h$ be a positive function on $\Omega$. A kernel $K$ on $\Omega\times\Omega$ is $h$-regular of type $l$ if there exists $a>0$ such that for all $\epsilon\in (0,a)$, we have
\begin{equation}\label{71}
\int_{\Omega}\left|K(z ;\zeta )\right|h^{-\epsilon}(\zeta)dV(\zeta )\lesssim h^{-\epsilon-l}{(z)}.
\end{equation}
\end{de}
Now we are ready to state our main theorem:
\begin{thm}
Let $\rho$ be a defining function for $\Omega\subseteq \mathbb C^n$ and let $U^{\alpha}\mathbb\subseteq \mathbb C^{n+k} $ be as in (\ref{00}). Suppose the Bergman kernel $K_{\Omega}$ satisfies the following two properties:
	\begin{enumerate}
		\item $K_{\Omega}$ is $(-\rho)$-regular of type $0$.
		\item $D^{\beta}_zK_{\Omega}(z;\bar{\zeta})$ is $(-\rho)$-regular of type $|\beta|$ whenever $|\beta|\leq k$.
	\end{enumerate}
	Then the Bergman projection is bounded on $L^p(U^{\alpha})$ for $p\in (1,\infty)$.
\end{thm}
We note that Assumption (1) implies that the Bergman projection on $\Omega$ is bounded in $L^p$ for $1<p<\infty$. See Schur's lemma in Section 3. Using estimates for derivatives of the Bergman kernel from \cite{McNeal2,McNeal1,NRSW,PS,CD}, one can show that $D^{\beta}_zK_{\Omega}$ is $(-\rho)$-regular of type $|\beta|$ for all $\beta\in \mathbb N^n$ in {classes of domains} previously mentioned.
In Theorem 1.2, we only require $D^{\beta}_zK_{\Omega}$ to be $(-\rho)$-regular of type $|\beta|$ for all $\beta$ such that $|\beta|\leq k$.

In Section 2, we recall the technique in \cite{Zhenghui} {relating the Bergman kernels of initial domains to those of their successors}. In Section 3, we discuss several lemmas and state Theorem 3.3. {This result is used to prove Theorem 1.2 via Schur's lemma}. We prove Theorem 3.3 in Section 4.

I would like to acknowledge John D'Angelo, Jeff McNeal, Brett Wick and the referee for their suggestions and comments.
\section{A formula for computing the Bergman kernel}
In this section we recall {a construction} from \cite{Zhenghui}{, which produces the Bergman kernel of various higher dimensional successors of an initial domain}. We start with an initial domain $\Omega$ and construct a class of domains  $U^{\alpha}(\Omega)$ by introducing new parameters $\alpha$ to $\Omega$. 

The technique in \cite{Zhenghui} consists of the following 4 steps:
\begin{enumerate}
	\item start with the kernel function $K_{\Omega}$ on the initial domain.
	\item construct a function on $U^{\alpha}(\Omega)\times U^{\alpha}(\Omega)$ by evaluating $K_{\Omega}$ at a point off the diagonal.
	\item define a specific differential operator (depending on $\alpha$).
	\item apply the operator in Step (3) to the function in Step (2), obtaining $K_{U^{\alpha}(\Omega)}$.
\end{enumerate}
The point at which we evaluate in Step (2) and the operator in Step (3) are independent of the initial domain $\Omega$, {but they depend on the parameters $\alpha$}. 

We {recall} in the definition below the notion of ``{complete Reinhardt}'' for the symmetry property the initial domain must satisfy.
\begin{de}
	A domain {$\Omega\subseteq\mathbb C^{n}$} is called {complete Reinhardt} in $(z_1,\dots,z_n)$ if the containment $(z_1,\dots,z_n)\in \Omega$ implies the containment
	$$\{(\lambda_1z_1,\dots,\lambda_n z_n):
	|\lambda_j|\leq 1 \:\:{\rm{{for}}}\:\: 1\leq j\leq n\}\subseteq\Omega.$$ 
\end{de}
Let $\Omega\subseteq \mathbb C^{n}$ be a {complete Reinhardt} domain in $(z_1,\dots,z_n)$. For {$\alpha\in \mathbb R^{n}_+$ and $w\in \mathbb B^k$}, set 
\begin{equation}\label{f}
f_{\alpha}\left(z,w\right)=\left(\frac{z_1}{(1-\|w\|^2)^{\frac{\alpha_1}{2}}},\dots,\frac{z_n}{(1-\|w\|^2)^{\frac{\alpha_n}{2}}}\right).
\end{equation}
The successor $U^\alpha(\Omega)$ is defined by 
{\begin{equation}\label{U}
	U^{\alpha}(\Omega)=\{(z,w)\in \mathbb C^{n}\times \mathbb B^k:f_{\alpha}(z,w)\in \Omega,\|w\|<1\}.
	\end{equation}}
For fixed $w\in \mathbb B^k$, let $U^{\alpha}_w(\Omega)$ denote the slice domain $\{z\in \mathbb C^{n}:(z,w)\in U^{\alpha}\}$ of $U^{\alpha}$. {We will often write $U^{\alpha}_{w}$ to denote $U^{\alpha}_w(\Omega)$. Since the mapping $f_{\alpha}(\cdot,w):z\mapsto f_\alpha(z,w)$ is a biholomorphism from $U^{\alpha}_w(\Omega)$ onto $\Omega$, the kernel on $U^{\alpha}_w(\Omega)$ can be obtained from $K_{\Omega}$.}

The main result in \cite{Zhenghui} relates the Bergman kernel on $U^{\alpha}_w(\Omega)$ to $K_{U^{\alpha}}$. {To state this result, we need a few more {notational definitions}.}
Let $I$ denote the identity operator. We define $D_{U^{\alpha}}$ to be the differential operator:
\begin{equation}\label{D}
D_{U^{\alpha}}=\frac{(1-\|\eta\|^2)^{|\alpha|}}{\pi^k(1- \langle w,\eta\rangle)^{1+k+|\alpha|}}\prod_{l=1}^{k}\left(lI+\sum_{j=1}^{n}\alpha_j\left(I+z_j\frac{\partial}{\partial z_j}\right)\right).
\end{equation}

Let $h(z,w,\eta)$ denote the following:
\begin{equation}\label{h}
h\left(z,w,\eta\right)=\left(z_1\left(\frac{1-\|\eta\|^2}{1-\langle w, \eta\rangle}\right)^{\alpha_1},\dots,z_n\left(\frac{1-\|\eta\|^2}{1-\langle w,\eta\rangle}\right)^{\alpha_n}\right).
\end{equation}

The formula for $K_{U^{\alpha}}$ in \cite{Zhenghui} can be expressed as follows:
\begin{thm}For {$(z,w;\zeta,\eta)\in U^{\alpha}\times U^{\alpha}$}, let $D_{U^{\alpha}}$ and $h(z,w,\eta)$  be as (\ref{D}) and (\ref{h}). Then
{\begin{equation}\label{3}
	K_{U^{\alpha}}\left(z,w;\bar \zeta,\bar \eta\right)=D_{U^{\alpha}}K_{U^{\alpha}_{\eta}}\left(h(z,w,\eta);\bar\zeta\right).
	\end{equation}}
\end{thm}
\section{Lemmas and Theorem 3.3}
The proof of Theorem 1.2 uses the following variant of Schur's {lemma}. See \cite{EM} for its proof. 
\begin{lem}[Schur's Lemma]
	Let $\Omega$ be a domain in $\mathbb C^n$ and let $K$ be a non-negative measurable function on $\Omega\times\Omega$. Let $\mathcal K$ be the integral operator with kernel $K$. Suppose there exists a positive auxiliary function $h$ on $\Omega$,  and a number $a>0$ such that for all $\epsilon\in (0,a)$, the following two inequalities hold:
	\begin{enumerate}
		\item $\mathcal K(h^{-\epsilon})(z)=\int_{\Omega}K(z,{\zeta})h({\zeta})^{-\epsilon}dV({\zeta})\lesssim h^{-\epsilon}(z)$,
		\item $\mathcal K(h^{-\epsilon})({\zeta})=\int_{\Omega}K(z,{\zeta})h(z)^{-\epsilon}dV(z)\lesssim h^{-\epsilon}({\zeta})$.
		\end{enumerate} 
		Then $\mathcal K$ is a bounded operator on $L^p(\Omega)$, for all $p\in (1,\infty)$.
		\end{lem}
	
		{We will take the function $K(z,\zeta)$ from Lemma 3.1 to be the absolute Bergman kernel $|K_{\Omega}(z;\bar \zeta)|$.} Inequalities (1) and (2) in the lemma are equivalent since {$K_{\Omega}(z;\bar \zeta)=\overline{K_{\Omega}(\zeta,\bar z)}$}. The $L^p$ boundedness of the corresponding operator $\mathcal K$ then implies the $L^p$ boundedness of the Bergman projection. To show that the Bergman projection on $\Omega$ is $L^p$ bounded for $p\in (1,\infty)$, it suffices to find an auxiliary function $h$ as in Lemma 3.{1} and show that $K_{\Omega}$ is $h$-regular of type $0$. {In many cases,} one can choose $h$ to be the distance function to the boundary.
	
			From now on we let $\Omega$ be a smooth bounded {complete Reinhardt} domain in $\mathbb C^n$. 	
			On such a domain $\Omega$, a defining function with {several useful symmetry properties} can be chosen.
			\begin{lem}
				Let $\Omega\subseteq\mathbb C^n$ be a smooth {complete Reinhardt} domain.   Then there exists a defining function $\rho$ of $\Omega$ satisfying the following properties: 
				\begin{enumerate}[label=(\alph*)]
					\item $\rho$ is smooth in a neighborhood of the boundary $\mathbf b \Omega$.
					\item   If $|z_j|=|\zeta_j|$ for $1\leq j\leq n$, then $\rho(z )=\rho(\zeta )$ 
					\item If $|z_j|\leq|\zeta_j|$ for $1\leq j\leq n$, then $\rho(z )\leq \rho(\zeta )$.
					\item For $1\leq j\leq n$, $z_j\rho_{z_j}(z)\geq0$.
					\item If $z\in \mathbf b\Omega$, then $\sum_{j=1}^{n}z_j\rho_{z_j}(z)>0$.
				\end{enumerate}
			\end{lem}
			\begin{proof}
				Set $\rho$ to be the function defined by the distance between $z$ and $\mathbf b\Omega$:
				$$\rho(z)=\begin{cases}
				\text{-dist$(z,\mathbf b\Omega)$}& z\in\Omega\\
				\text{dist$(z,\mathbf b\Omega)$}& z\notin\Omega
				\end{cases}.$$
				{Then {property} (a) is true for any domain $\Omega$ with smooth boundary. {Properties} (b) and (c) also hold since $\Omega$ is complete Reinhardt}. 
			 Consider polar coordinates $z_j=t_je^{i\theta_j}$ for $1\leq j\leq n$. Since $\rho$ is invariant under the rotation in each coordinates, we have:
				\begin{align}\label{02}
				0=\frac{\partial}{\partial \theta_j}\rho\left(t_1e^{i\theta_1},\dots,t_ne^{i\theta_n}\right)=i\left(z_j\rho_{z_j}(t_1e^{i\theta_1},\dots,t_ne^{i\theta_n})-\bar{z}_j\rho_{\bar{z}_j}(t_1e^{i\theta_1},\dots,t_ne^{i\theta_n})\right).\end{align}
				The monotonicity of $\rho$ in $|z_j|$ implies:
				\begin{align}\label{03}
				0\leq t_j\frac{\partial}{\partial t_j}\rho\left(t_1e^{i\theta_1},\dots,t_ne^{i\theta_n}\right)=z_j\rho_{z_j}(t_1e^{i\theta_1},\dots,t_ne^{i\theta_n})+\bar{z}_j\rho_{\bar{z}_j}(t_1e^{i\theta_1},\dots,t_ne^{i\theta_n}). 
				\end{align}
				Combining these two formulas yields Property (d). 
				
				To prove Property (e), it suffices to show that $\sum_{j=1}^{n}z_j\rho_{z_j}(z)\neq0$ on $\mathbf b\Omega$. Suppose not. Then there exists some $z\in\mathbf b\Omega$ such that $z_j\rho_{z_j}(z)=0$ for all $j$. Let $\mathcal A$ denote the set of indices $j$ such that $z_j=0$ and let $\mathcal B$ denote the complement of $\mathcal A$ in $\{1,\dots, n\}$. Then $\rho_{z_j}(z)$ equals 0 for all $j\in \mathcal A$. Since the gradient of $\rho$ does not vanish on $\mathbf b\Omega$, there exists an index $j_0\in \mathcal B$ such that $\rho_{z_{j_0}}(z)\neq 0$. Thus $z_{j_0}$ equals 0. The fact that $z_{j_0}=0$ and Property (c) then imply that $z$ is a local min for $\rho(z)$ in the $z_{j_0}$ direction. This contradicts $\rho_{z_{j_0}}(z)\neq 0$. Therefore the sum $\sum_{j=1}^{n}z_j\rho_{z_j}(z)$ does not vanish on the boundary.
			\end{proof}
				
			The crucial estimates for Theorem 1.2 arise from the following theorem:
				\begin{thm}
					Let $\Omega\subseteq \mathbb C^n$ be a smooth {complete Reinhardt} domain with a defining function $\rho$. For $\alpha\in \mathbb R^n_+$, let $f_{\alpha}$ and $U^{\alpha}$  be as {(\ref{f}) and (\ref{U})}. If $D^{\beta}_zK_{\Omega}$ is $(-\rho)$-regular whenever $ |\beta|\leq k$,
					then $K_{U^{\alpha}}$ is $\left((1-\|w\|^2)(-\rho\circ f_\alpha)\right)$-regular of type $0$.
				\end{thm}
				{We give a proof for Theorem 3.3 in Section 4.  Theorem 3.3  implies Theorem 1.2. Indeed, the kernel $K_{U^{\alpha}}$ being $\left((1-\|w\|^2)(-\rho\circ f_\alpha)\right)$-regular of type $0$ implies that the Bergman projection on $U^{\alpha}$ is bounded in $L^p$ for $p\in (1,\infty)$. }
				
			We end this section by referencing several estimates needed in the proof of Theorem 3.3. See for example \cite{Zhu}.
			\begin{lem}Let $\sigma$ denote Lebesgue measure on the unit sphere $\mathbb S^k\subset\mathbb C^k$.
				For $\epsilon<1$ and $w\in\mathbb B^k$, let 
				\begin{equation}\label{**}
				a_{\epsilon,\delta}(w)=\int_{\mathbb B^k}\frac{(1-\|\eta\|^2)^{-\epsilon}}{|1-\langle w,\eta\rangle |^{1+k-\epsilon-\delta}}dV(\eta),
				\end{equation}
				and let
				\begin{equation}\label{ }
				b_\delta(w)=\int_{\mathbb S^k}\frac{1}{|1-\langle w,\eta\rangle |^{k-\delta}}d\sigma(\eta).
				\end{equation}
				Then \begin{enumerate}
					\item for $\delta>0$, both $a_{\epsilon,\delta}$ and $b_{\delta}$ are bounded on $\mathbb B^k$.
					\item for $\delta=0$, both $a_{\epsilon,\delta}(w)$ and $b_{\delta}(w)$ are comparable to the function $-\log(1-\|w\|^2)$.
					\item for $\delta<0$, both $a_{\epsilon,\delta}(w)$ and $b_{\delta}(w)$ are comparable to the function $(1-\|w\|^2)^{\delta}$.
				\end{enumerate}
			\end{lem}

		\section{Proof of Theorem 3.3}
\begin{proof}[Proof of Theorem 3.3]
	Recall that for each multi-index $\beta$,   $D^{\beta}_z$ is the differential operator $(\frac{\partial}{\partial z_1})^{\beta_1}\cdots (\frac{\partial}{\partial z_n})^{\beta_n}$.   Then $D_{U^{\alpha}}$ in the previous section can be regarded as a sum of $D^{\beta}_z$:
{\begin{equation}
	D_{U^{\alpha}}=\frac{(1-\|\eta\|^2)^{|\alpha|}}{\pi^k(1- \langle w,\eta\rangle)^{1+k+|\alpha|}}\left(\sum_{|\beta|\leq k} c_{\beta}z^{\beta}D^\beta_z\right),
	\end{equation}
	where $c_\beta$ are fixed constants.}
	  
	{The main goal in this proof is to show the following inequality: 
		\begin{equation}\label{59}
		\int_{U^{\alpha}}\left|K_{U^{\alpha}}(z ,w;\bar\zeta ,\bar \eta)\right|\left(-\rho\left(f_\alpha(\zeta,\eta) \right)(1-\|\eta\|^2)\right)^{-\epsilon}dV\lesssim \left(-\rho\left(f_\alpha(z,w) \right)(1-\|w\|^2)\right)^{-\epsilon}.
		\end{equation}}
	To estimate the integral
			\begin{equation}\label{6}
			\int_{U^{\alpha}}\left|K_{U^{\alpha}}(z ,w;\bar\zeta ,\bar \eta)\right|\left(-\rho\left(f_\alpha(\zeta,\eta) \right)(1-\|\eta\|^2)\right)^{-\epsilon}dV,
			\end{equation}
			we use the formula in Theorem 2.2.
			Substituting (\ref{3}) into the integral in (\ref{6}) yields
			\begin{align}\label{7}
			&\int_{U^{\alpha}}\left|K_{U^{\alpha}}(z ,w;\bar\zeta ,\bar \eta)\right|\left(-\rho\left(f_\alpha(\zeta,\eta) \right)(1-\|\eta\|^2)\right)^{-\epsilon}dV\nonumber\\=&\int_{U^{\alpha}}\left|D_{U^{\alpha}}K_{U^{\alpha}_{\eta}}\left(h(z,w,\eta) ;\bar\zeta \right)\right|\left(-\rho\left(f_\alpha(\zeta,\eta) \right)(1-\|\eta\|^2)\right)^{-\epsilon}dV.
			\end{align}
		
			We set
			\begin{align}
			I_\beta&=\frac{c_{\beta}(1-\|\eta\|^2)^{|\alpha|}}{(1- \langle w,\eta\rangle)^{1+k+|\alpha|}}z^{\beta}D^\beta_z,\nonumber
			\end{align}
			and
			\begin{align}\label{8}
			J_{\beta}&=\int_{U^{\alpha}}\left|I_\beta K_{U^{\alpha}_{\eta}}\left(h(z,w,\eta) ;\bar\zeta \right)\right|\left(-\rho\left(f_\alpha(\zeta,\eta) \right)(1-\|\eta\|^2)\right)^{-\epsilon}dV.
			\end{align}
			By the triangle inequality, we have
			{\begin{align}\label{9}
			\int_{U^{\alpha}}\left|D_{U^{\alpha}}K_{U^{\alpha}_{\eta}}\left(h(z,w,\eta) ;\bar\zeta \right)\right|\left(-\rho\left(f_\alpha(\zeta,\eta) \right)(1-\|\eta\|^2)\right)^{-\epsilon}dV\leq \sum_{|\beta|\leq k} J_{\beta}.
			\end{align}}
			Therefore it suffices to prove that $J_\beta\lesssim \left(-\rho\left(f_\alpha(z,w) \right)(1-\|w\|^2)\right)^{-\epsilon}$ for each $\beta$.
			
			The integral $J_\beta$ equals
			\begin{align}\label{10}
			c_\beta\int_{U^{\alpha}}\left|\frac{(1-\|\eta\|^2)^{|\alpha|}}{(1- \langle w,\eta\rangle)^{1+k+|\alpha|}}z^{\beta}D^\beta_z K_{U^{\alpha}_{\eta}}\left(h(z,w,\eta) ;\bar\zeta \right)\right|\left(-\rho\left(f_\alpha(\zeta,\eta) \right)(1-\|\eta\|^2)\right)^{-\epsilon}dV.
			\end{align}
		{	In order to use $(-\rho)$-regularity assumptions of $D^\beta K_\Omega$ for estimating (\ref{10}), we need to write $D^\beta_z K_{U^{\alpha}_{\eta}}$ in (\ref{10})  in terms of $D^\beta_z K_{\Omega}$ and transform (\ref{10}) into an integral on $\mathbb B^k\times \Omega$.
			
		Recall the mapping $f_\alpha(\cdot,\eta)$ from $\ref{f}$ defined by 
			\begin{equation}
				f_\alpha(\cdot,\eta ):z\mapsto\left(\frac{z_1}{(1-\|\eta\|^2)^{\alpha_1/2}},\cdots,\frac{z_n}{(1-\|\eta\|^2)^{\alpha_n/2}}\right).
			\end{equation}
		It	is a biholomorphism from $U^\alpha_\eta$ onto $\Omega$. Hence we can write the kernel function $K_{U^{\alpha}_\eta}$ in terms of $K_\Omega$ using the biholomorphic transformation formula:
		\begin{equation}\label{111}
		K_{U^{\alpha}_\eta}(z;\bar\zeta)=(1-\|\eta\|^2)^{-|\alpha|}K_\Omega(f_\alpha(z,\eta),\overline{f_\alpha(\zeta,\eta)}).
		\end{equation}}
			Applying (\ref{111}) to (\ref{10}) yields:
			\begin{align}\label{11}
			J_\beta&=c_\beta\int_{U^{\alpha}}\left|\frac{z^{\beta}{D^{\beta}_z}K_{\Omega}\left(h^{\prime}(z,w,\eta) ;\overline{f_{\alpha}(\zeta,\eta)} \right)}{|1- \langle w,\eta\rangle|^{1+k+|\alpha|}}\right|\left(-\rho\left(f_\alpha(\zeta,\eta) \right)(1-\|\eta\|^2)\right)^{-\epsilon}dV,
			\end{align}
			where $h^{\prime}(z,w,\eta)=\left(\frac{z_1(1-\|\eta\|^2)^{\alpha_1/2}}{(1-\langle w,{\eta}\rangle)^{\alpha_1}},\dots, \frac{z_n(1-\|\eta\|^2)^{\alpha_n/2}}{(1-\langle w, {\eta}\rangle)^{\alpha_n}}\right)$. 
			
		 By Substituting $t_j=\frac{\zeta_j}{(1-\|\eta\|^2)^{\alpha_j/2}}$ for $1\leq j\leq n$ to (\ref{11}), we transform $J_\beta$ into an integral on $\mathbb B^k \times \Omega$:
			\begin{equation}\label{12}
			J_\beta=c_\beta\int_{\mathbb B^k}\int_{\Omega}\left|\frac{z^{\beta}{D^\beta_z} K_{\Omega}\left(h^{\prime}(z,w,\eta) ;\bar t \right)}{(1-\|\eta\|^2)^{\epsilon-|\alpha|}|1- \langle w,\eta\rangle |^{1+k+|\alpha|}}\right|\left(-\rho\left(t \right)\right)^{-\epsilon}dV(t )dV(\eta).
			\end{equation}
		For $1\leq j\leq n$, let $D_j$ denote the partial derivative $\frac{\partial}{\partial z_j}$.	Since 
			\begin{align}
			D_jK_{\Omega}\left(h^{\prime}(z,w,\eta) ;\bar t \right)&=\frac{\partial h^{\prime}_j}{\partial z_j}\frac{\partial}{\partial h^{\prime}_j}K_{\Omega}\left(h^{\prime}(z,w,\eta) ;\bar t \right)\nonumber\\&=\frac{(1-\|\eta\|^2)^{\alpha_j/2}}{(1-\langle w,\eta\rangle)^{\alpha_j}}\frac{\partial}{\partial h^{\prime}_j}K_{\Omega}\left(h^{\prime}(z,w,\eta) ;\bar t \right),
			\end{align}
			applying the $(-\rho)$-regularity of ${D^{\beta}_z}K_\Omega$ to the inner integral in (\ref{12}) yields
			\begin{equation}\label{13}
			J_\beta\lesssim \int_{\mathbb B^k}\left|\frac{z^{\beta}\left(-\rho\left(h^{\prime}(z,w,\eta) \right)\right)^{-\epsilon-|\beta|}}{(1-\|\eta\|^2)^{\epsilon-|\alpha|-\alpha\cdot\beta/2}|1- \langle w,\eta\rangle |^{1+k+\alpha\cdot(\mathbf 1+\beta)}}\right|dV(\eta).
			\end{equation}
			Here we use the notation $\alpha\cdot\beta$ to denote $\sum_{j=1}^{n}\alpha_j\beta_j$ and use the notation $\mathbf 1$ to denote the multi-index $(1,\dots,1)\in \mathbb N^n$.
			When $\beta=\mathbf 0$, we have 
			\begin{equation}\label{131}
			J_\mathbf 0\lesssim \int_{\mathbb B^k}\left|\frac{\left(-\rho\left(h^{\prime}(z,w,\eta) \right)\right)^{-\epsilon}}{(1-\|\eta\|^2)^{\epsilon-|\alpha|}|1- \langle w,\eta\rangle |^{1+k+|\alpha|}}\right|dV(\eta).
			\end{equation}
			{Since $w,\eta\in \mathbb B^k$,} the triangle inequality and Cauchy-Schwarz inequality {imply}
			$$\left|\frac{z_j(1-\|\eta\|^2)^{\alpha_j/2}}{(1-\langle w,{\eta}\rangle)^{\alpha_j}}\right|\leq\left|\frac{z_j(1-\|\eta\|^2)^{\alpha_j/2}}{(1-\|w\|^2)^{\alpha_j/2}(1-\|\eta\|^2)^{\alpha_j/2}}\right|=\left|\frac{z_j}{(1-\|w\|^2)^{\alpha_j/2}}\right|.$$
			Therefore, Property {(c)} in Lemma {3.2} implies:
			\begin{align}\label{14}
			J_\mathbf 0&\lesssim \int_{\mathbb B^k}\left|\frac{\left(-\rho\left(h^{\prime}(z,w,\eta) \right)\right)^{-\epsilon}}{(1-\|\eta\|^2)^{\epsilon-|\alpha|}|1- \langle w,\eta\rangle|^{1+k+|\alpha|}}\right|dV(\eta)\nonumber\\&\leq\left(-\rho(f_{\alpha}(z,w))\right)^{-\epsilon}\int_{\mathbb B^k}\frac{(1-\|\eta\|^2)^{-\epsilon+|\alpha|}}{|1- \langle w,\eta\rangle|^{1+k+|\alpha|}}dV(\eta).
			\end{align}
			For $w,\eta\in\mathbb B^k$, we have 
			\begin{equation}\label{*}
			\frac{1-\|\eta\|^2}{|1-\langle w,{\eta}\rangle |}\leq\frac{1-\|\eta\|^2}{1-|\langle w,{\eta}\rangle |}< \frac{1-\|\eta\|^2}{1-\|{\eta}\|}<2.
			\end{equation} 
			Applying this inequality and Lemma 3.4 to (\ref{14}) yields the inequality we need for $J_0$:
			\begin{align}\label{15}
			J_\mathbf 0&\lesssim \left(-\rho(f_{\alpha}(z,w))\right)^{-\epsilon}\int_{\mathbb B^k}\frac{(1-\|\eta\|^2)^{-\epsilon}}{|1- \langle w,\eta\rangle|^{1+k}}dV(\eta)\nonumber
			\\&\lesssim \left(-\rho(f_{\alpha}(z,w))\right)^{-\epsilon}(1-\|w\|^2)^{-\epsilon}.
			\end{align}
			{For the case $\beta\neq \mathbf 0$, we recall the integral we need to estimate:}
			\begin{equation}\label{16}
			\int_{\mathbb B^k}\left|\frac{z^{\beta}\left(-\rho\left(h^{\prime}(z,w,\eta) \right)\right)^{-\epsilon-|\beta|}}{(1-\|\eta\|^2)^{\epsilon-|\alpha|-\alpha\cdot\beta/2}|1- \langle w,\eta\rangle |^{1+k+\alpha\cdot(\mathbf 1+\beta)}}\right|dV(\eta).\end{equation}
			{After rewriting} the integral in spherical coordinates $\eta=rt$ with $r\in[0,1]$ and $t\in \mathbb S^k$, we would like to {write $\left(-\rho\left(h^{\prime}(z,w,\eta) \right)\right)^{-\epsilon-|\beta|}$ in terms of the $|\beta|$-th order derivative of $\left(-\rho\left(h^{\prime}(z,w,rt) \right)\right)^{-\epsilon}$ in $r$.}
		These derivatives vanish at the point $\eta=w$ and hence  are relatively small when compared with $(-\rho)^{-\epsilon-|\beta|}$ . To deal with this problem, we need to move the vanishing point $\eta=w$ to the origin.
		
			When $w=\mathbf 0$, we keep (\ref{16}) the same. When $w\in \mathbb B^k-\{\mathbf 0\}$, we set $$\varphi_w(z)=\frac{w-P_w(z)-s_wQ_w(z)}{1-\langle z,w\rangle},$$
			where $s_w=\sqrt{1-\|w\|^2}$, $P_w(z)=\frac{\langle z,w\rangle}{\|w\|^2}w$ and $Q_w(z)=z-\frac{\langle z,w\rangle}{\|w\|^2}w$. {Then $\varphi_w$ is the automorphism of $\mathbb B^k$ that sends $\mathbf 0$ to $w$ and satisfies $\varphi_w\circ\varphi_w=id$. We use this $\varphi_w$ to send the point $\eta=w$ to the origin.}
			Setting ${\tau}=\varphi_w(\eta)$, then we have
			\begin{align}\label{22}
			\eta&=\varphi_w({\tau}),
			\\\label{221}1-\langle \eta,w \rangle&=\frac{1-\|w\|^2}{1-\langle {\tau},w\rangle},
			\\\label{222}1-\|\eta\|^2&=\frac{(1-\|w\|^2)(1-\|{\tau}\|^2)}{|1-\langle {\tau},w\rangle|^2},
			\\\label{223}dV(\eta)&=\left(\frac{1-\|w\|^2}{|1- \langle {\tau},w\rangle|^2}\right)^{k+1}dV({\tau}).
			\end{align}
			Substituting (\ref{22}), (\ref{221}), (\ref{222}) and (\ref{223}) into the integral (\ref{16}) yields
			\begin{align}
			&\int_{\mathbb B^k}\left|\frac{z^{\beta}\left(-\rho\left(h^{\prime}(z,w,\eta) \right)\right)^{-\epsilon-|\beta|}}{(1-\|\eta\|^2)^{\epsilon-|\alpha|-\alpha\cdot\beta/2}|1- \langle w,\eta\rangle |^{1+k+\alpha\cdot(\mathbf 1+\beta)}}\right|dV(\eta)\nonumber
			\\=&\int_{\mathbb B^k}\frac{|z^{\beta}|(\frac{(1-\|w\|^2)(1-\|{\tau}\|^2)}{|1-\langle {\tau},w\rangle|^2})^{\alpha\cdot\beta/2-\epsilon+|\alpha|}}{\left|\frac{1-\|w\|^2}{1- \langle {\tau},w\rangle}\right|^{1+k+\alpha\cdot(\mathbf 1+\beta)}\frac{(1-\|w\|^2)^{-k-1}}{|1-\langle {\tau},w \rangle|^{-2(k+1)}}}\left(-\rho\left(h^{\prime}(z,w,\varphi_w({\tau})\right)\right)^{-\epsilon-|\beta|}dV({\tau}).
			\end{align}
			Canceling terms in the integral gives
			\begin{align}\label{23}
			\int_{\mathbb B^k}\frac{|z^{\beta}|(1-\|{\tau}\|^2)^{\alpha\cdot\beta/2-\epsilon+|\alpha|}}{|1-\langle {\tau},w\rangle|^{1+k-2\epsilon+|\alpha|}(1-\|w\|^2)^{\alpha\cdot\beta/2+\epsilon}}\left(-\rho\left(h^{\prime}(z,w,\varphi_w({\tau})) \right)\right)^{-\epsilon-|\beta|}dV({\tau}),
			\end{align}
			which is consistent with (\ref{16}) when $w=\mathbf 0$.
			Applying inequality (\ref{*}) to (\ref{23}) and using the fact that $\frac{|z^{\beta}|}{(1-\|w\|^2)^{\alpha\cdot\beta/2}}$ is bounded on $\Omega$, we obtain the following inequality:
			\begin{align}\label{230}
			&\int_{\mathbb B^k}\frac{|z^{\beta}|(1-\|{\tau}\|^2)^{\alpha\cdot\beta/2-\epsilon+|\alpha|}}{|1-\langle {\tau},w\rangle|^{1+k-2\epsilon+|\alpha|}(1-\|w\|^2)^{\alpha\cdot\beta/2+\epsilon}}\left(-\rho\left(h^{\prime}(z,w,\varphi_w({\tau})) \right)\right)^{-\epsilon-|\beta|}dV({\tau})\nonumber
			\\\lesssim& \int_{\mathbb B^k}\frac{(1-\|{\tau}\|^2)^{\alpha\cdot\beta/2-\epsilon+|\alpha|}}{|1-\langle {\tau},w\rangle|^{1+k-2\epsilon+|\alpha|}(1-\|w\|^2)^{\epsilon}}\left(-\rho\left(h^{\prime}(z,w,\varphi_w({\tau})) \right)\right)^{-\epsilon-|\beta|}dV({\tau})\nonumber
			\\\lesssim& \int_{\mathbb B^k}\frac{\left(-\rho\left(h^{\prime}(z,w,\varphi_w({\tau})) \right)\right)^{-\epsilon-|\beta|}}{|1-\langle {\tau},w\rangle|^{1+k-\epsilon}(1-\|w\|^2)^{\epsilon}}dV({\tau}).
			\end{align}
		We set $l(z,w,{\tau})=(l_1(z,w,{\tau}),\dots,l_n(z,w,{\tau}))$ where 
			$$l_j(z,w,{\tau})=\left|h^{\prime}_j\left(z,w,\varphi_w({\tau})\right)\right|=\left|\frac{z_j\left(\frac{(1-\|w\|^2)(1-\|{\tau}\|^2)}{|1-\langle {\tau},w\rangle|^2}\right)^{\alpha_j/2}}{\left(\frac{1-\|w\|^2}{1-\langle {\tau},w\rangle}\right)^{\alpha_j}}\right|=\frac{|z_j|(1-\|{\tau}\|^2)^{\alpha_j/2}}{(1-\|w\|^2)^{\alpha_j/2}}.$$
			Then Lemma 3.2 implies that
			$\rho(h^{\prime}(z,w,\varphi_w({\tau})) )=\rho(l(z,w,{\tau}) ),$
			and the integral in the last line of (\ref{230}) becomes
			\begin{equation}\label{232}
			\int_{\mathbb B^k}\frac{\left(-\rho\left(l(z,w,{\tau}) \right)\right)^{-\epsilon-|\beta|}}{|1-\langle {\tau},w\rangle|^{1+k-\epsilon}(1-\|w\|^2)^{\epsilon}}dV({\tau}).
			\end{equation}
			Rewriting (\ref{232}) using spherical coordinates {$\tau=rt$ with $r\in [0,1)$ and $t\in \mathbb S^k$} yields:
			\begin{align}\label{240}
			c_k\int_{0}^{1}r^{2k-1}\int_{\mathbb S^k}\frac{\left(-\rho\left(l(z,w,r t\right)\right)^{-\epsilon-|\beta|}}{|1-\langle rt,w\rangle|^{1+k-\epsilon}(1-\|w\|^2)^{\epsilon}}d\sigma (t)dr,
			\end{align}
			where $c_k$ is a constant depending on the dimension $k$.
			
			By Property (e) in Lemma 3.2,  there exists an open neighborhood $\mathcal U$ of $\mathbf b\Omega$ such that for any $z\in \mathcal U$,\begin{equation}\label{0} \sum_{j=1}^{n}z_j\rho_{z_j}(z)>c,\end{equation} for some positive $c$. 
			For $\delta>0$, let $\bar\Omega_{\delta}$ denote the set
			\begin{equation}\label{***}\{z\in \mathbb C^n:\rho((1+\delta)^{\alpha_1/2}z_1,\dots,(1+\delta)^{\alpha_n/2}z_n)\leq0\}.
			\end{equation}
			Then there exists a constant $\delta_0>0$ such that $\Omega-\mathcal U\subseteq \bar \Omega_{\delta_0}$.
			Since $\bar\Omega_{\delta_0}$ is compact in $\Omega$,  we have $(-\rho(z))^{-1}<C$ in $\bar\Omega_{\delta_0}$ for some constant $C$. Let $\mathcal U_0$ denote the set $\bar{\Omega}_{\delta_0}$, and let $\mathcal U_1$ denote the set $\Omega-\mathcal U_0$. Then on $\mathcal U_1$, inequality (\ref{0}) still holds. For $t\in\mathbb S^k$ and $j=0,1$, set $$U_j=\{r\in [0,1]:l(z,w,rt)\in\mathcal U_j\}.$$
			Here $U_j$'s are well-defined for any $t\in \mathbb S^k$: for fixed $z$ and $w$, the value of $l(z,w,rt)$ only depends on $r$ and $\|t\|$.
			For each $U_j$, we set 
			\begin{align}\label{250}
			\mathcal  I^{\beta}_j&=\int_{U_j}r^{2k-1}\int_{\mathbb S^k}\frac{\left(-\rho\left(l(z,w,r t\right)\right)^{-\epsilon-|\beta|}}{|1-\langle rt,w\rangle|^{1+k-\epsilon}(1-\|w\|^2)^{\epsilon}}d\sigma (t)dr.
			\end{align} 
		{We claim that $\mathcal  I^{\beta}_j\lesssim \left((-\rho)(l(z,w,t))(1-|w|^2)\right)^{-\epsilon}$ for each $j$. Then by having
			\begin{equation}
		J_\beta\lesssim \mathcal I^{\beta}_0+\mathcal I^{\beta}_1\lesssim \left((-\rho)(l(z,w,t))(1-|w|^2)\right)^{-\epsilon},
		\end{equation} 
		we complete the proof.}
			
			We first consider $\mathcal I^\beta_0$. Since $(-\rho(l(z,w,rt)))^{-1}<C$ for $r\in U_0$, {we have}
			\begin{align}\label{24}
			\mathcal I^{\beta}_0\lesssim\int_{U_0}r^{2k-1}\int_{\mathbb S^k}\frac{1}{|1-\langle rt,w\rangle|^{1+k-\epsilon}(1-\|w\|^2)^{\epsilon}}d\sigma (t)dr.
			\end{align}
			Applying Lemma 3.4 to the inner integral  of (\ref{24}) yields:
			\begin{equation}\label{241}
			\int_{\mathbb S^k}\frac{1}{|1-\langle rt,w\rangle|^{1+k-\epsilon}(1-\|w\|^2)^{\epsilon}}d\sigma (t)\lesssim(1-\|w\|^2)^{-\epsilon}(1-r^2\|w\|^2)^{\epsilon-1}.
			\end{equation}
			{Then {inequality} (\ref{241}) }gives the desired estimate for $\mathcal I^{\beta}_0$:
			\begin{align}\label{25}
			\mathcal I^{\beta}_0\lesssim&\int_{U_0}r^{2k-1}(1-\|w\|^2)^{-\epsilon}(1-r^2\|w\|^2)^{\epsilon-1}dr\nonumber\\\lesssim&\int_{U_0}r^{2k-1}(1-\|w\|^2)^{-\epsilon}(1-r^2)^{\epsilon-1}dr\nonumber\\\lesssim &(1-\|w\|^2)^{-\epsilon}\nonumber\\\lesssim&\left((-\rho)(l(z,w,t))(1-|w|^2)\right)^{-\epsilon}.
			\end{align}
			
			Now we turn to $\mathcal I^\beta_1$. When $r\in U_1$, we have $l(z,w,rt)\in \mathcal U_1$ and 
			\begin{equation}
			\sum_{j=1}^{n}l_j(z,w,rt)\rho_{z_j}(l(z,w,rt))>c.
			\end{equation}
			For such an $r$, $\frac{\partial}{\partial r} \left((-\rho)^{-\epsilon-|\beta|+1}(l(z,w,rt))\right)$ is controlled from below by $(-\rho)^{-\epsilon-|\beta|}(l(z,w,rt))$:
			\begin{align}\label{29}
			&-\frac{\partial}{\partial r}(-\rho(l(z,w,rt)))^{-\epsilon-|\beta|+1}\nonumber\\=&2(\epsilon+|\beta|-1)(-\rho)^{-\epsilon-|\beta|}(l(z,w,rt))\sum_{j=1}^{n}\frac{\alpha_jr|z_j|(1-r^2)^{\alpha_j/2-1}}{(1-\|w\|^2)^{\alpha_j/2}}\rho_{z_j}(l(z,w,rt))\nonumber\\\gtrsim&\frac{r(-\rho)^{-\epsilon-|\beta|}(l(z,w,rt))}{(1-r^2)}\sum_{j=1}^nl_j(z,w,rt)\rho_{z_j}(l(z,w,rt))\nonumber\\\gtrsim&\frac{r(-\rho)^{-\epsilon-|\beta|}(l(z,w,rt))}{(1-r^2)}.
			\end{align}
			Applying (\ref{29}), (\ref{*}) and Lemma 3.4 to (\ref{250}) then yields:
			\begin{align}\label{253}
			\mathcal I^{\beta}_1\lesssim&-\int_{U_j}r^{2k-2}\int_{\mathbb S^k}\frac{(1-r^2)\frac{\partial}{\partial r}\left(-\rho\left(l(z,w,r t\right)\right)^{-\epsilon-|\beta|+1}}{|1-\langle rt,w\rangle|^{1+k-\epsilon}(1-\|w\|^2)^{\epsilon}}d\sigma (t)dr\nonumber
			\\\lesssim&-\int_{U_j}r^{2k-2}\int_{\mathbb S^k}\frac{\frac{\partial}{\partial r}\left(-\rho\left(l(z,w,r t\right)\right)^{-\epsilon-|\beta|+1}}{|1-\langle rt,w\rangle|^{k-\epsilon}(1-\|w\|^2)^{\epsilon}}d\sigma (t)dr\nonumber
			\\\lesssim&-(1-\|w\|^2)^{-\epsilon}\int_{U_j}r^{2k-2}{\frac{\partial}{\partial r}\left(-\rho\left(l(z,w,r t\right)\right)^{-\epsilon-|\beta|+1}}dr.
			\end{align}
			Since for fixed $z$ and $w$, the point $l(z,w,0)$ is closest to $\mathbf b\Omega$, we may assume $U_1=[0,r_0]$ where $r_0$ depends on both $z$ and $w$.
			
		{When $k=1$, integrating the last line of (\ref{253}) by parts yields
			\begin{align}\label{269}
			 &-(1-\|w\|^2)^{-\epsilon}\int_{U_1}r^{2k-2}{\frac{\partial}{\partial r}\left(-\rho\left(l(z,w,r t)\right)\right)^{-\epsilon-|\beta|+1}}dr\nonumber
			 \\=&-(1-\|w\|^2)^{-\epsilon}\int_{0}^{r_0}{\frac{\partial}{\partial r}\left(-\rho\left(l(z,w,r t)\right)\right)^{-\epsilon-|\beta|+1}}dr\nonumber
			 \\=&-\frac{r^{2k-2}{\left(-\rho\left(l(z,w,r t)\right)\right)^{-\epsilon-|\beta|+1}}\Big|_0^{r_0}}{(1-\|w\|^2)^{\epsilon}}
			 \end{align}
Noting that  $k=1$ also implies $-\epsilon-|\beta|+1\geq -\epsilon-k+1=-\epsilon$, we have
			 	\begin{align}\label{28}
			 	&-(1-\|w\|^2)^{-\epsilon}\int_{U_1}{\frac{\partial}{\partial r}\left(-\rho\left(l(z,w,r t)\right)\right)^{-\epsilon}}dr\nonumber
			 	\\\leq&\frac{{\left(-\rho\left(l(z,w,0)\right)\right)^{-\epsilon}}+{\left(-\rho\left(l(z,w,r_0t)\right)\right)^{-\epsilon}}}{(1-\|w\|^2)^{\epsilon}}.
			 	\end{align}
			 	By its definition, the point $l(z,w,r_0t)$ is in $\mathcal {U}_0$. Therefore $\left(-\rho\left(l(z,w,r_0t\right)\right)^{-\epsilon-|\beta|+1}\lesssim 1$ and the desired estimate follows:
			 	\begin{equation}\label{260}
			 	\mathcal I^{\beta}_1=-\int_{U_1}\frac{\frac{\partial}{\partial r}\left(-\rho\left(l(z,w,r t)\right)\right)^{-\epsilon}}{(1-\|w\|^2)^{\epsilon}}dr\lesssim\left(-\rho\left(f_{\alpha}(z,w)\right)\right)^{-\epsilon}(1-\|w\|^2)^{-\epsilon}.
			 	\end{equation}
			
			When $k>1$, integrating the last line of (\ref{253}) by parts yields
				\begin{align}\label{27}
				&-(1-\|w\|^2)^{-\epsilon}\int_{U_1}r^{2k-2}{\frac{\partial}{\partial r}\left(-\rho\left(l(z,w,r t)\right)\right)^{-\epsilon-|\beta|+1}}dr\nonumber
				\\=&-(1-\|w\|^2)^{-\epsilon}\int_{0}^{r_0}r^{2k-2}{\frac{\partial}{\partial r}\left(-\rho\left(l(z,w,r t)\right)\right)^{-\epsilon-|\beta|+1}}dr\nonumber
				\\=&-\frac{r^{2k-2}{\left(-\rho\left(l(z,w,r t)\right)\right)^{-\epsilon-|\beta|+1}}\Big|_0^{r_0}}{(1-\|w\|^2)^{\epsilon}}+\int_{0}^{r_0}(2k-2)r^{2k-3}\frac{{\left(-\rho\left(l(z,w,r t)\right)\right)^{-\epsilon-|\beta|+1}}}{(1-\|w\|^2)^{\epsilon}}dr.
				\end{align}
			{The numerator of the first term in the last line equals ${r_0^{2k-2}(-\rho(l(z,w,r_0t)))^{-\epsilon-|\beta|+1}}$}, which is also controlled by a constant. Thus it remains to show that
			\begin{equation}\label{30}
			\int_{0}^{r_0}r^{2k-3}\frac{{\left(-\rho\left(l(z,w,r t)\right)\right)^{-\epsilon-|\beta|+1}}}{(1-\|w\|^2)^{\epsilon}}dr\lesssim \left(-\rho\left(f_{\alpha}(z,w)\right)\right)^{-\epsilon}(1-\|w\|^2)^{-\epsilon}.
			\end{equation}}
			Applying (\ref{29}) to the left hand side of (\ref{30}) gives
			$$\int_{0}^{r_0}r^{2k-3}\frac{{\left(-\rho\left(l(z,w,r t)\right)\right)^{-\epsilon-|\beta|+1}}}{(1-\|w\|^2)^{\epsilon}}dr\lesssim-\int_{0}^{r_0}r^{2k-4}\frac{\frac{\partial}{\partial r}{\left(-\rho\left(l(z,w,r t)\right)\right)^{-\epsilon-|\beta|+2}}}{(1-\|w\|^2)^{\epsilon}}dr.$$
			This together with (\ref{27}) implies that for $k>1$
			\begin{align}\label{31}
			&-\int_{0}^{r_0}r^{2k-2}\frac{\frac{\partial}{\partial r}{\left(-\rho\left(l(z,w,r t)\right)\right)^{-\epsilon-|\beta|+1}}}{(1-\|w\|^2)^{\epsilon}}dr\nonumber\\\lesssim&-\int_{0}^{r_0}r^{2k-4}\frac{\frac{\partial}{\partial r}{\left(-\rho\left(l(z,w,r t)\right)\right)^{-\epsilon-|\beta|+2}}}{(1-\|w\|^2)^{\epsilon}}dr.
			\end{align}
		{Since (\ref{31}) holds whenever $|\beta|\leq k$, we have for $0<s\leq k$
				\begin{align}\label{310}
			&-\int_{0}^{r_0}r^{2k-2}\frac{\frac{\partial}{\partial r}{\left(-\rho\left(l(z,w,r t)\right)\right)^{-\epsilon-s+1}}}{(1-\|w\|^2)^{\epsilon}}dr\nonumber\\\lesssim&-\int_{0}^{r_0}r^{2k-4}\frac{\frac{\partial}{\partial r}{\left(-\rho\left(l(z,w,r t)\right)\right)^{-\epsilon-s+2}}}{(1-\|w\|^2)^{\epsilon}}dr.
			\end{align}
			Repeated use of inequality (\ref{310}) then gives
			\begin{align}\label{32}
			&-\int_{0}^{r_0}r^{2k-2}\frac{\frac{\partial}{\partial r}{\left(-\rho\left(l(z,w,r t)\right)\right)^{-\epsilon-|\beta|+1}}}{(1-\|w\|^2)^{\epsilon}}dr\nonumber\\\lesssim&-\int_{0}^{r_0}r^{2k-4}\frac{\frac{\partial}{\partial r}{\left(-\rho\left(l(z,w,r t)\right)\right)^{-\epsilon-|\beta|+2}}}{(1-\|w\|^2)^{\epsilon}}dr\nonumber\\&\;\;\;\;\;\;\;\;\;\;\;\;\;\;\;\;\;\;\;\;\;\;\;\;\;\;\;\;\;\nonumber\vdots\\\lesssim&-\int_{0}^{r_0}r^{2k-2|\beta|}\frac{\frac{\partial}{\partial r}{\left(-\rho\left(l(z,w,r t)\right)\right)^{-\epsilon}}}{(1-\|w\|^2)^{\epsilon}}dr
			\end{align}
			Noting that $r^{2k-2|\beta|}$ is bounded on $[0,r_0]$, we have
				\begin{equation}\label{320}
			-\int_{0}^{r_0}r^{2k-2|\beta|}\frac{\frac{\partial}{\partial r}{\left(-\rho\left(l(z,w,r t)\right)\right)^{-\epsilon}}}{(1-\|w\|^2)^{\epsilon}}dr\leq-\int_{0}^{r_0}\frac{\frac{\partial}{\partial r}{\left(-\rho\left(l(z,w,r t)\right)\right)^{-\epsilon}}}{(1-\|w\|^2)^{\epsilon}}dr.
				\end{equation}
Applying inequality (\ref{260}) to (\ref{320}) then yields
	 \begin{equation}\mathcal I^{\beta}_1=-\int_{0}^{r_0}r^{2k-2}\frac{\frac{\partial}{\partial r}{\left(-\rho\left(l(z,w,r t)\right)\right)^{-\epsilon-|\beta|+1}}}{(1-\|w\|^2)^{\epsilon}}dr\lesssim\left(-\rho\left(f_{\alpha}(z,w)\right)\right)^{-\epsilon}(1-\|w\|^2)^{-\epsilon},
	 \end{equation}
	 which completes the proof.}
		\end{proof}
{\paragraph{\bf Remark}As in the proof of Thereom 3.3, we can obtain an $L^p$ regularity result for the Bergman projection on more generalized domains which are generated from $\Omega$ by iterating the construction of $U^{\alpha}$ from (\ref{U}).
		
		Set $\alpha=(\alpha^{(1)},\dots,\alpha^{(l)})\in\mathbb R^n_+\times\dots\times\mathbb R^n_+$ where each $\alpha^{(j)}$ is in $\mathbb R^n_+$. Let $k_1,\dots,k_l$ be $l$ positive integers. The successor  $\mathbf U(\Omega)$ is defined by
		\begin{equation}\label{33}
				\mathbf U(\Omega)=\{(z,w_1,w_2\cdots,w_l)\in \mathbb C^n\times\mathbb B^{k_1}\times\cdots\times\mathbb B^{k_l}:(\mathbf f_{\alpha}(z,w_1,\dots,w_l))\in \Omega\},
		\end{equation}
		where 
		\begin{equation}\label{34}
		\mathbf f_{\alpha}\left(z,w_1,\dots,w_l\right)=\left(\frac{z_1}{\prod_{j=1}^{l}(1-\|w_j\|^2)^{\frac{\alpha^{(j)}_1}{2}}},\dots,\frac{z_n}{\prod_{j=1}^{l}(1-\|w_j\|^2)^{\frac{\alpha^{(j)}_n}{2}}}\right).
		\end{equation}
			Suppose $\Omega\subseteq \mathbb C^n$ is a smooth {complete Reinhardt} domain with defining function $\rho$ and  $D^{\beta}_zK_\Omega(z;\bar{\zeta})$ is $(-\rho)$-regular of type $|\beta|$ for $0\leq |\beta|\leq \sum_{j=1}^{l}k_j$.
			Then the Bergman projection on $\mathbf U(\Omega)$ is $L^p$ bounded for all $1< p< \infty$.
		The proof of this statement is similar to the proof for the first successor. We omit it here.}
\bibliographystyle{alpha}
\bibliography{1}

\end{document}